\date{}

\documentclass[10pt]{amsart}
\usepackage{latexsym,amsmath,amsfonts,amscd,amssymb}
\usepackage{graphics}
\newtheorem{theorem}{Theorem}[section]

\newtheorem{proposition}[theorem]{Proposition}
\newtheorem{corollary}[theorem]{Corollary}

\newtheorem{definition}[theorem]{Definition}
.

\theoremstyle{remark}
\newtheorem{remark}[theorem]{Remark}

\newcommand{\Ker}{\operatorname{Ker}}

\newcommand{\cA}{{\mathcal A}}

\newcommand{\cD}{{\mathcal D}}
\newcommand{\cE}{{\mathcal E}}

\newcommand{\cI}{{\mathcal I}}
\newcommand{\cJ}{{\mathcal J}}

\newcommand{\cO}{{\mathcal O}}

\newcommand{\CC}{{\mathbb C}}
\newcommand{\DD}{{\mathbb D}}

\newcommand{\NN}{{\mathbb N}}
\newcommand{\PP}{{\mathbb P}}
\newcommand{\QQ}{{\mathbb Q}}

\newcommand{\ZZ}{{\mathbb Z}}

\renewcommand{\a}{\alpha}
\renewcommand{\b}{\beta}

\newcommand{\g}{\gamma}

\newcommand{\eps}{\epsilon}

\begin{document}
\title[The e\~ne product over a commutative ring]{The  e\~ne product over a commutative ring}


\author[R. P\'{e}rez-Marco]{Ricardo P\'{e}rez-Marco}
\address{CNRS, IMJ-PRG UMR 7586, Universit\'e Paris Cit\'e \\ 
B\^at. Sophie Germain, 75205 Paris, France}

\email{ricardo.perez.marco@gmail.com}

\maketitle

{ \centerline{\sc Abstract}}

\bigskip

\begin{minipage}{15cm}
\noindent

We define the e\~ne product for the multiplicative group of polynomials and formal power series with coefficients on 
a commutative ring and unitary constant coefficient\footnote{Note (June 2026): The e\~ne product is a twisted form of the multiplication in the Big Witt ring. In 2019, when this article was first made public,  the author didn't know about the Big Witt ring. The presentation given here provides new formulas and a new analytic insights absent in the classical literature. It also provides a novel and natural construction
of the Big Witt ring (see \cite{BPMR}).}.
This defines a commutative ring structure where multiplication is
the additive structure and the e\~ne product is the multiplicative one. For polynomials over $\CC$, the e\~ne product 
acts as a multiplicative convolution of their divisor. We study its algebraic properties, its relation to symmetric 
functions on an infinite number of variables, to tensor products, and Hecke operators. The exponential
linearizes also the e\~ne product.  The e\~ne product extends to rational functions and formal meromorphic functions.
We also study the analytic properties over $\CC$, and for entire functions. The e\~ne product respects  
Hadamard-Weierstrass factorization and is related to the Hadamard product. 
The e\~ne product plays a central role in predicting the phenomenon of the ``statistics on Riemann zeros'' for Riemann 
zeta function and general Dirichlet $L$-functions discovered by the author in \cite{PM2}. It also gives reasons to believe in the Riemann 
Hypothesis as explained in \cite{PM3}.

\end{minipage}


\tableofcontents


\section{Preliminaries.} \label{sec:preliminaries}

All rings considered in this article are commutative unitary rings.
We consider a commutative ring $(A,+,.)$  with unit $1\in A$ and 
the associated local ring of formal power series $A[[X]]$ with 
coefficients in $A$. When the series are meant to be convergent power series
with complex coefficients we shall use the variable $z$ instead of $X$.
The definition of the
e\~ne product is valid for arbitrary rings. For some formulas involving
logarithms or exponentials we require the ring $A$ to contain $\QQ$,
i.e. $A$ is a $\QQ$-algebra, as we
need to divide by integers. On the other hand, the universal polynomial
formulas with integer coefficients remain valid for arbitrary rings.

\begin{proposition} Let $\cA=1+XA[[X]]$. The multiplication of formal power series
is an internal operation in $\cA$ and $(\cA , .)$ is an abelian 
group with $1$ as neutral element.
\end{proposition}

Sometimes we denote $\cA_A$ to indicate the coefficient ring $A$.
We recall some basic facts about the logarithmic derivative 
and the exponential.

\begin{definition}The logarithmic derivative $\cD : (\cA, .) \to (A[[X]], +)$
$$
f \mapsto \cD(f)=f'/ f
$$
is a morphism of groups.

For $f,g \in \cA$, $a \in A$, we have
\begin{align*}
\cD (f.g) &= \cD(f) +\cD(g)  \ ,\\
\cD(1) &= 0 \ ,\\
\cD (f(a X)) &=a \cD(f)(a X)  \ .
\end{align*}
If $f\in A[[X]]$ and if $f(0)\in A^\times$ is invertible then the logarithmic derivative $\cD (f)$ is defined
and takes values in $A[[X]]$ and $\cD (a f)=\cD(f)$.

\end{definition}

The morphism $\cD$ is an isomorphism when $\QQ\subset A$, $\Ker \cD =\{1\}$.
From now on in this section we assume that $\QQ\subset A$, i.e. $A$ is a $\QQ$-algebra.
\footnote{Alternatively, we can work in full generality
with the enveloping algebra of formal power series with variables labelled by $A$, $\QQ [[(X_a)_{a\in A}]]$ (see \cite{BPMR} and \cite{Bou} Chapter 4 for calculus with formal power series in an infinite number of variables). This ring is a $\QQ$-algebra and specializes to $A$ when
we identify the variables $X_a=a$ and quotient by the ideal of relations in $A$. Thus all universal polynomial formulas with integer coefficients remain valid when quotienting into $A$. Another way to avoid the restrictions of rings with
non-zero characteristic is to replace the formulas with exponentials
with logarithmic derivatives, but it is important to keep
formulas in exponential form.}

\begin{definition} The
exponential map $\exp :(XA[[X]] , +) \to (\cA , .)$
$$
f \mapsto \exp f=e^f=\sum_{n=0}^{+\infty} \frac{f^n}{n!}
$$
is an isomorphism of groups.
\end{definition}

Note that the inverse map is given by the logarithm morphism
$\log : (\cA , .) \to (XA[[X]] , +)$,
$$
\log(f) = \sum_{n=1}^{+\infty} \frac{(-1)^{n+1}}{n} (f-1)^n \ .
$$

The exponential map and the logarithmic derivative map do
factor the derivative operator. More precisely, the derivative operator
is the group isomorphism $D: (XA[[X]], +) \to (A[[X]],+)$
$$
f \mapsto D(f)=f'=\frac{d}{dX} f(X)
$$
and it factors as
$$
D=\cD \circ \exp \ .
$$

A related natural operator is the exponential logarithmic
derivative.

\begin{definition}
The exponential logarithmic derivative
is the group isomorphism $\cD_{\exp } : \, (\cA , .) \to (\cA, .)$
$$
f \mapsto \cD_{\exp } (f)=e^{X \cD (f)} \ .
$$
\end{definition}

\section{The e\~ne ring. Definition and first properties.} \label{sec:ene}

In this section we comsider an arbitrary commutative ring $A$.
We define the e\~ne product on $\cA$.

Let $(X_1, \dots, X_n)$ and $(Y_1, \dots, Y_m)$ be two sets of variables.

For $p \leq n, m$ we define
$$
\Sigma_p^{n\otimes m} =
\Sigma_{(i_1,j_1),\dots,(i_p,j_p)}(X_{i_1}Y_{j_1})\dots(X_{i_p}Y_{j_p})
\in \ZZ [X_1,\dots,X_n,Y_1,\dots,Y_m]
$$
where the sum runs over all elements of $(\{1,\ldots , n\}\times \{1,\ldots , m\})^p$.
We refer to \cite{PM1} or \cite{PM-MG-B-J}  for qualitative and quantitative generalizations of 
the following proposition using the theory of symmetric functions (see \cite{PM1} and \cite{Bou} for the generalization of the fundamental theorem of symmetric functions in an infinite
uncountable number of variables and to \cite{PM-MG-B-J} for explicit bounds).
Below we provide a direct proof (that will appear more natural after section 4).

\begin{proposition}\label{prop:sym}
For $p\leq \min\{n,m\}$, there exists a universal polynomial 
$Q_p \in \ZZ [X_1, \ldots ,X_p , Y_1, \ldots , Y_p]$ independent
of $n\geq p$ and $m\geq p$ such that
$$
\Sigma_p^{n\otimes m} =
Q_{p}(\Sigma_1^X,\dots,\Sigma_p^X,\Sigma_1^Y,\dots,\Sigma_p^Y)
$$
where the $\Sigma_k^X$ and $\Sigma_k^Y$ are the corresponding symmetric
functions in each set of variables.

We have 
$$
(-1)^p Q_p(X_1,\ldots ,X_p , Y_1, \ldots , Y_p)=- p X_p Y_p + P_p(X_1,
\ldots X_p,  Y_1, \ldots , Y_p)
$$
where $P_p$ does not contain any monomial $X_pY_p$, and 
the weight on the $X$'s and $Y$'s of each monomial of $P_p$ is $2p$. The 
weight of $X_{i_1}^{n_1}\ldots X_{i_p}^{n_p}Y_{j_1}^{m_1}\ldots Y_{j_p}^{m_p}$
being $n_1 i_1+\ldots + n_p i_p+ m_1 j_1 \ldots m_p j_p$.
 
\end{proposition}

\begin{proof} We carry out the computations in the ring of formal power series
$\QQ[[(X_i)_{1\leq i\leq n}, (Y_j)_{1\leq j\leq m}, Z]]$
Consider the polynomials 
\begin{align*}
f(Z) &= \prod_{i=1}^n (1-X_i Z) =1+\sum_{k=1}^{n } \Sigma_k^X \ Z^k \ ,\\
g(Z) &= \prod_{j=1}^m (1-Y_j Z) =1+\sum_{k=1}^{m } \Sigma_k^Y \ Z^k \ .
\end{align*}
Now in the same way
$$
\prod_{i,j} (1-X_iY_j Z)=1+\sum_{k=1}^{p } \Sigma_k^{n\otimes m} \ Z^k +
\cO (Z^{p+1}) \ .
$$
Observe that 
\begin{align*}
f(Z)&=\exp \left ( \log (1+\sum_{k=1}^{n } \Sigma_k^X \ Z^k )\right )\\
&=\exp \left ( \sum_{k=1}^{+\infty } K_k (\Sigma_1^X , \ldots , \Sigma_k^X) Z^k \right )
\end{align*}
where $K_k (U_1, \ldots , U_k)$ is a polynomial with rational coefficients of weight $k$ on the
$U$ variables and $K_k (U_1, \ldots , U_k)=U_k+L_k(U_1, \ldots , U_{k-1})$.
Also 
\begin{align*}
f(Z)&=\exp \left ( \sum_i \log (1-X_i Z) \right )\\
&=\exp \left ( -\sum_{k=1}^n  \frac{1}{k} \left (\sum_i X_i^k \right )
 Z^k \right )
\end{align*}

Observe now that 
$$
\left (\sum_i X_i^k \right ).\left (\sum_j Y_j^k \right )=
\sum_{i,j} (X_iY_j)^k
$$
thus
$$
\exp \left (-\sum_{k=1}^{+\infty } k K_k (\Sigma_1^X , \ldots , \Sigma_k^X)
K_k (\Sigma_1^Y , \ldots , \Sigma_k^Y)  Z^k\right )=
1+\sum_{k=1}^{p } \Sigma_k^{n\otimes m} Z^k +\cO (Z^{p+1}) \ .
$$
But also the expansion on power series on $Z$ gives
\begin{align*}
&\exp \left (-\sum_{k=1}^{+\infty } k K_k (\Sigma_1^X , \ldots , \Sigma_k^X)
K_k (\Sigma_1^Y , \ldots , \Sigma_k^Y)  Z^k\right )= \\
&=1+\sum_{k=1}^p (-1)^k Q_k(\Sigma_1^X , \ldots , \Sigma_k^X ,
\Sigma_1^Y , \ldots , \Sigma_k^Y ) Z^k +\cO (Z^{p+1})
\end{align*}
where $Q_k$ is the polynomial with rational coefficients and the required properties.
It only remains to check that $Q_k$ has indeed
integer coefficients and not just rational coefficients.
The polynomials $\Sigma_p^{k\otimes k} \in \ZZ[X_1,\ldots , X_k, Y_1, \ldots Y_k]$
are symmetric in the two group of variables $(X_1, \dots, X_k)$ and $(Y_1, \dots, Y_k)$. Using
the Fundamental Theorem of Symmetric functions in the ring $\ZZ[X_1, \dots, X_k]$ we have that
$\Sigma_p^{k\otimes k}$ is a polynomial with coefficients in $\ZZ[X_1, \dots, X_k]$  in the
variables $(\Sigma_1^Y , \ldots , \Sigma_k^Y )$. Each coefficient in $\ZZ[X_1, \dots, X_k]$ is a
symmetric polynomial in the variables $(X_1,\ldots , X_k)$.
Applying a second time the Fundamental Theorem of Symmetric functions  in the ring $\ZZ$
to each polynomial coefficient we get that $Q_k$ has integer coefficients.
\end{proof}

\begin{definition}
For any ring $A$, the e\~ne product of $f,g \in \cA$,
\begin{align*}
f(X)&=1+a_1 X+a_2 X^2+\ldots \cr 
g(X)&=1+b_1 X+b_2 X^2+\ldots
\end{align*}
is defined by
$$
f \star g (X)=1+c_1 X+c_2 X^2+\ldots 
$$
where for $n\geq 1$, $c_n$ is given  by
$$
c_n =(-1)^n Q_n(a_1, \ldots , a_n , b_1, \ldots , b_n) \ ,
$$
where $Q_n \in \ZZ [X_1,\ldots ,X_n, Y_1,\ldots , Y_n]$ are the polynomials
from  Proposition \ref{prop:sym}.
\end{definition}

The following is inmediate from the definition.

\begin{proposition}
The e\~ne product is an internal operation of $\cA$.
If $A\subset \CC$ and $(\a_i)$ and $(\beta_j)$ are the 
roots of two polynomials $f$ and $g$ then the roots of $f\star g$ 
are $(\alpha_i \beta_j)_{i,j}$.
\end{proposition}

Note that the coefficient $c_n$ only depends on the 
coefficients of order $\leq n$. This operation, contrary to the 
sum and product, is not pointwise geometric. It is geometric 
in the roots. We give some explicit  formulas for the firsts coefficients.

\begin{proposition}
We have
\begin{align*}
c_1 &= -a_1 b_1 \ , \\
c_2 &= -2a_2 b_2 +a_2 b_1^2+a_1^2 b_2 \ , \\
c_3 &= -3a_3 b_3 +3a_3 b_1 b_2 -a_3 b_1^3 +3a_1 a_2 b_3 -a_1 a_2 b_1 b_2
-a_1^3 b_3 \ . 
\end{align*}
\end{proposition}

Now the main property follows.

\begin{theorem} \textbf{(Distributivity of the e\~ne product)} \label{thm:distributivity}
The e\~ne product $\star$
is distributive with 
respect to the multiplication. If 
$f,g,h \in \cA$ then
$$
(f.g)\star h =(f\star h).(g \star h) \ .
$$
\end{theorem}

\begin{proof}
The $n$-th order coefficient of $(f.g)\star h $ (resp. 
$(f\star h).(g \star h)$) is a polynomial with 
integer coefficients on the coefficients of order $\leq n$ of 
$f$, $g$ and $h$. Thus, by universality, it is enough to establish the identity 
when $A=\CC$ and when $f$, $g$ and $h$ are polynomials. 
Because in such case the polynomials
with integer coefficients giving the expressions of order $n$ on 
both sides will agree on an open set of $\CC^{n^3}$ thus are 
equal (we must choose $f$, $g$ and $h$ of degree larger than $n$).

 If $(\a_i)$, $(\b_j)$ and $(\g_k)$ are respectively the zeros 
of $f$, $g$ and $h$ counted with multiplicity  
then the zeros counted with multiplicity of $(f.g)\star h $
and $(f\star h).(g \star h)$ are $(\a_i\g_k)_{i,k}
\cup (\b_j \g_l)_{j,l}$. 
Thus these two polynomial functions have the same zeros, and constant value $1$,
so they must be equal, and the result follows. 
\end{proof}

\begin{theorem}The set $(\cA , ., \star )$ is 
a commutative ring with zero $1\in \cA$ and unity 
$1-X \in \cA$.
More precisely, we have 
\begin{itemize}
\item  $(\cA , .)$ is an abelian group.

\item  (Distributivity) For $f,g,h \in \cA$, 
$(f.g)\star h =(f\star h).(g \star h)$.

\item  (Associativity)  For $f,g,h \in \cA$, 
$(f \star g) \star  h = f \star (
g \star  h)$ \ .

\item  (Commutativity) For $f,g \in \cA$,
$f \star g=g \star f$.

\item  (Unit) For $f \in \cA$,
$f \star (1-X)=(1-X)\star f=f$.
\end{itemize}
\end{theorem}

\begin{proof}
We already  proved the distributive property. The other properties
follow in the same way.
\end{proof}
We have in the e\~ne ring $(\cA , ., \star )$ the usual identities
in commutative rings:

\begin{corollary}
For $f,g \in \cA$ and $n\geq 1$  we have

\begin{itemize}
\item  $f\star 1=1\star f=1$.

\item  $f\star (1/g)=(1/f)\star g =
\frac{1}{f\star g}$.

\item  For $n\in \ZZ$, $f\star g^n=f^n\star g =
(f\star g)^n$.

\item  ${1\over f} \star {1\over g}=f\star g$.

\item  Newton binomial formula.
$$
\left (f . g \right )^{\star n}=\prod_{k=0}^n 
\left ( f^{\star (n-k)}\star g^{\star k} 
\right )^{\binom{n} {k}} \ .
$$
\end{itemize}

\end{corollary}

We have also some additional properties that are proved 
as in Theorem \ref{thm:distributivity}.

\begin{theorem}
We have 
\begin{itemize}
\item {$(1)$} If $f,g \in \cA$ and $a \in A$
we have
$$
f(a X) \star g(X) =f(X) \star g(a X) =(f\star g) (aX) \ .
$$
In particular,
$$
(1-a X)\star f(X)=f(a X) \ .
$$

\item {$(2)$} For $f, g\in \cA$ and $k\geq 1$ positive integer,
$$
f(X^k)\star g(X^k)=\left ((f\star g)(X^k) \right )^k \ .
$$

\item {$(3)$} For $f, g\in \cA$ and $k, l\geq 1$ positive integers
with $k\wedge l=1$,
$$
f(X^k)\star g(X^l)=(f\star g)(X^{kl}) \ .
$$
\end{itemize}

\end{theorem}

\begin{proof} The proof of (1) is clear.
For the proof of (2), we consider polynomials $f(z)$ and $g(z)$
with complex coefficients. Observe  that if the roots of $f$
(resp. $g$) are the $(\a_i)$ (resp. $(\b_j)$), then the 
roots of $f(z^k)$ (resp. $g(z^k)$, $f\star g (z^k)$,
$f(z^k)\star g(z^k)$)
are the $(\eps \a_i^{1/k})$
(resp. $(\eps \b_j^{1/k})$, $(\eps \a_i^{1/k} \b_j^{1/k})$,
$(\eps \eps' \a_i^{1/k} \b_j^{1/k})$ ) 
where $\eps$ (and $\eps'$) runs over the
group $\mathbb U_k$ of $k$-roots of $1$. Now, the map
$\mathbb U_k^2 \to \mathbb  U_k$, $(\eps, \eps')\mapsto \eps \eps'$
is $k$-to-$1$ and the result follows by universality of the formulas.

The proof of (3) is similar observing that the map
$\mathbb U_k \times \mathbb U_l \to \mathbb U_{kl}$, $(\eps, \eps')
\mapsto \eps \eps'$, is a bijection when $k\wedge l=1$.
\end{proof}

\begin{theorem}We assume that $A\subset\CC$.
If $f,g\in \cA$ are polynomials or entire functions of order $<1$
with respective zeros $(\a_i)$ and $(\b_j)$, we have
$$
f\star g (z)=\prod_{i,j} \left (1-{z\over \a_i \b_j} \right )
=\prod_j f \left ( {z\over \b_j } \right ) = 
\prod_i  g \left ( {z\over \a_i }\right ) \ .
$$
\end{theorem}

This last result extends to arbitrary entire functions for each product that is converging.

\section{Main formula and first applications.} \label{sec:main}

We assume $\QQ\subset A$ in this section.
The following fundamental relation relates the exponential, the
logarithmic derivative and the e\~ne product.

\begin{theorem}\textbf{(Main Formula).}
 For $f,g \in \cA$ we have
$$
\exp \left ( X\cD(f\star g) \right ) =g\star \exp (X\cD(f))
=f\star \exp (X\cD(g)) \ .
$$

Or, in terms of the exponential logarithmic derivative,

$$
\cD_{\exp } (f\star g) =f\star \cD_{\exp } (g)=
\cD_{\exp } (f)\star g \ .
$$
\end{theorem}

\begin{proof}
We observe again
that it is enough to prove the result for $f$ and $g$ polynomials
with complex coefficients. We consider $f$ and $g$ polynomials with respective sets
of zeros $(\a_i)$ and $(\b_j)$. Observe that 
$$
(f\star g)(z)=\prod_{i,j} \left (1-{z\over \a_i \b_j}\right ) 
=\prod_j f\left ({z\over \b_j } \right ) \ ,
$$
thus
$$
\cD( f\star g) (z)=\sum_j {1\over \b_j } (\cD f)(z/\b_j) \ ,
$$
so
$$
z\cD( f\star g) (z)=\sum_j {z\over \b_j } (\cD f)(z/\b_j) \ ,
$$
and using Theorem 2.9
$$
e^{z \cD (f \star g) }=\prod_j e^{{z\over \b_j } (\cD f)(z/\b_j) }
=g\star e^{z\cD(f)(z)} \ .
$$
\end{proof}

\begin{corollary}
Let $f\in \cA$, $f(X)=1+f_1 X+\ldots $,
and $a\in A$. 
We have 
$$
f\star e^{aX} =e^{-af_1 X} \ .
$$
\end{corollary}

\begin{proof}
Put $g(X)=e^{aX}$ in the Main Formula. Observe that
$$
\cD (e^{aX}) =a \ .
$$
We get
$$
e^{X\cD (f \star e^{aX})}=f\star e^{X\cD (e^{aX})}
=f\star e^{aX} \ .
$$
Thus $F(X) =f\star e^{aX}$ satisfies the 
differential equation
$$
F=e^{X\cD (F)}=e^{X F'/ F} \ .
$$
We define $G=\log F \in XA[[X]]$ then $G'=F'/F$ and $G$ satisfies the 
differential equation
$$
G'={1\over X} G \ .
$$
If we write $G(X)=a_0X+a_1X^2+\ldots$ this means that for $n\geq 1$, 
$(n-1)  a_n=0$, so $a_n=0$.
Therefore,  the only formal solutions are $G(X)=a_0 X$ for some constant
$a_0 \in A$. So finally
$$
F(X)=f\star e^{aX}=e^{a_0 X} \ .
$$
To determine $a_0$, using the formula for $c_1$, we observe that
\begin{align*}
f\star e^{aX} &=(1+f_1X+\ldots )\star (1+aX+\ldots) \\
&= 1 -af_1 X+\ldots 
\end{align*}
and therefore $a_0=-f_1 a$.
\end{proof}

\bigskip

More generally we have the following result.

\begin{corollary}
Let $f\in \cA$, $f(X)=1+f_1 X+\ldots $,
$a\in A$, and $n\geq 1$ positive integer.
We have 
$$
f\star e^{aX^n} =e^{a \tilde Q_n(f_1,\ldots , f_n) X^n} \ ,
$$
where 
$$
\tilde Q_n(X_1,\ldots , X_n)=(-1)^n Q_n(X_1,\ldots , X_n, 0,\ldots, 0, Y)/Y
=-n X_n+ P_n( X_1, \ldots X_{n-1})
$$
is a polynomial vanishing when $X_1=X_2=\ldots =X_n=0$.
\end{corollary}

\begin{proof}

As before, using the main formula we get
\begin{align*}
e^{X \cD (f\star e^{a X^n})} &=f\star 
e^{X \cD (e^{a X^n} )} \\
&=f\star e^{naX^n} \\
&=\left ( f\star e^{aX^n} \right )^n 
\end{align*}
Then $F(X)=f\star e^{aX^n}$ satisfies the differential 
equation
$$
F^n =e^{XF'/F}
$$
Thus $G=\log F$ satisfies
$$
nG=XG'
$$
which has only formal solutions $G(X)=a_0 X^n$, $a_0\in A$. 
To determine the constant $a_0$ we write the first term of the 
expansion
\begin{align*}
f\star e^{a X^n} &=(1+f_1 X+\ldots )\star 
(1+a X^n +\ldots )\\
&=1+(-1)^n Q_n(f_1,\ldots, f_n, 0,\ldots , 0, a) X^n+\ldots 
\end{align*}
Thus 
$$
a_0=(-1)^n Q_n(f_1,\ldots, f_n, 0,\ldots , 0, a) =(-nf_n +\ldots ) a
$$
where the quantity between brackets is independent of $a$ 
and has monomials of weight $n$ on the coefficients $(f_i)$
(see proposition 2.1). 
\end{proof}

\begin{corollary}
For $n,m \geq 1$ positive integers, 
and $a,b \in A$, we have if $n\not= m$,
$$
e^{aX^n}\star e^{bX^m} =1
$$
and for $n=m$,
$$
e^{aX^n}\star e^{bX^n} =e^{-n ab X^n} \ .
$$
 \end{corollary}

\section{Exponential form and applications.} \label{sec:expo}

We assume in this section that $\QQ\subset A$.
We can cut short the previous discussions and formulas with the following 
key result. It shows that the e\~ne product operator $\star$ has
a very simple expression in exponential form, or, in other words, we have
the remarkable property that the exponential bilinearizes the e\~ne product.

\begin{theorem} \textbf{(Exponential form).} 
Let $f, g \in \cA$. Using the isomorphism given by the exponential
map, we can write
\begin{align*}
f&=e^F=e^{F_1 X+F_2 X^2+F_3 X^3 +\ldots} \\
g&=e^G=e^{G_1 X+G_2 X^2+G_3 X^3 +\ldots} 
\end{align*}
where $F,G \in A[[X]]$.

We have
\begin{align*}
f\star g &=\exp \left ( F_1 X+F_2 X^2+F_3 X^3 +\ldots \right )\star
\exp \left ( G_1 X+G_2 X^2+G_3 X^3 +\ldots \right ) \\
&=\exp \left ( -F_1 G_1 X -2 F_2 G_2 X^2
-3 F_3 G_3 X^3 +\ldots \right ) \  .
\end{align*}
We denote by $\star_e$ the exponential form of the e\~ne product
$$
F\star_e G= -F_1 G_1 X -2 F_2 G_2 X^2 -3 F_3 G_3 X^3 +\ldots \ .
$$
\end{theorem}

\begin{proof}

We simply use the distributivity of $\star$ and the 
previous corollary:
\begin{align*}
f\star g &=\exp\left (\sum_{i=1}^{+\infty} F_i X^i \right )
\star \exp\left (\sum_{j=1}^{+\infty} G_j X^j \right ) \cr
&=\left ( \prod_{i=1}^{+\infty } \exp (F_i X^i) \right ) \star
\left ( \prod_{j=1}^{+\infty } \exp (G_j X^j) \right ) \cr 
&= \prod_{i,j =1}^{+\infty} \exp (F_i X^i) \star \exp (G_j X^j) \cr
&= \prod_{i =1}^{+\infty} \exp (-i F_i G_i X^i) \cr
&=\exp \left ( -\sum_{i=1}^{+\infty} i F_i G_i X^i \right )
\end{align*}
\end{proof}


Using this formula we can now determine exactly which elements
of the ring $(\cA , . , \star)$ are divisors of zero (the zero
is the constant series $1$).

\begin{theorem}
The divisors of zero in the e\~ne-ring 
$(\cA , . , \star)$ are exactly those $f\in \cA$ such
that if we write
$$
f=e^F=\exp \left (\sum_{i=1}^{+\infty} F_i X^i \right )
$$
there is some coefficient $F_i$ that is $0$ or a divisor of $0$ 
in $A$. Thus, if $A$ has no zero divisors, only those $f \in \cA$ for
which some $F_i=0$ are divisors of $0$.

We remind that $\QQ\subset A$. The elements $f\in \cA$ that are not divisors of zero
are e\~ne-invertible, i.e. are 
units of the ring $\cA$, if and only if each one of its 
exponential coefficients has an inverse, the inverse
being
$$
g=e^G=\exp \left (\sum_{i=1}^{+\infty} G_i X^i \right )
$$
with 
$$
G_i = {1 \over i^2 } F_i^{-1} \ .
$$
\end{theorem}

\begin{proof}

The e\~ne neutral element $1-X$ has the exponential form
$$
1-X=\exp (\log (1-X))=\exp \left ( 
-\sum_{i=1}^{+\infty} {1\over i} X^i \right ) \ .
$$
and the result follows.
\end{proof}


\begin{remark}

Notice that when $A\subset \CC$, if $f$ has infinite radius
of convergence (i.e. it is an entire function) 
and is e\~ne-invertible, then its e\~ne-inverse has
zero radius of convergence.
 
\end{remark}

\section{Some e\~ne products.} \label{sec:prod}

The next result shows that when $A$ is a field, the e\~ne product of rational functions
is a rational function. But we have a more general result.

\begin{theorem}Let $A$ be an arbitrary unitary commutative ring. The e\~ne product leaves invariant
the multiplicative subgroup $(1+XA[X])/(1+XA[X]) \subset 1+XA[[X]]$ of
formal power series that are quotients of polynomials in $1+XA[X]$. We denote
this subgroup by $A_0(X)$. Thus
$A_0(X) \subset \cA$ is a subring of the e\~ne ring.
More precisely, let $R_1(X) , R_2(X) \in A_0(X)$ with
\begin{align*}
R_1(X) &= {P_1(X) \over Q_1(X) } \cr
R_2(X) &= {P_2(X) \over Q_2(X) }
\end{align*}
with $P_1(X) , P_2(X)  , Q_1(X)  , Q_2(X) \in 1+XA[X]$. Then
$$
R_1(X) \star R_2(X)={\left (P_1(X) \star P_2(X) \right ) .
\left (Q_1(X) \star Q_2(X) \right ) \over \left (P_1(X) \star Q_2(X) \right )
.\left (Q_1(X) \star P_2(X) \right )} \ .
$$
We can define the degree of  $R(X)={P(X)\over Q(X)}\in A_0(X)$ by
$$
\deg R(X)=\max (\deg P(X), \deg Q(X)) \ .
$$
Then we have
$$
\deg ( R_1(X) \star R_2(X) )=\deg R_1(X) . \deg R_2(X) \ .
$$
\end{theorem}

We observe that when $A\subset \CC$,
the zeros of $R_1(z)\star R_2(z)$ are the products of zeros of
$R_1$ and $R_2$ or the product of poles of $R_1$ and $R_2$. Also
the poles of $R_1(z)\star R_2(z)$ are the products of a pole and
a zero of $R_1$ and $R_2$. In short, we can write
\begin{align*}
{\hbox {\rm zero}}\ \star \ {\hbox {\rm zero}}\ &= \ {\hbox {\rm zero}} \ \\
{\hbox {\rm pole}}\ \star \ {\hbox {\rm pole}}\ &= \ {\hbox {\rm zero}} \ \\
{\hbox {\rm zero}}\ \star \ {\hbox {\rm pole}}\ &= \ {\hbox {\rm pole}} \ \\
{\hbox {\rm pole}}\ \star \ {\hbox {\rm zero}}\ &= \ {\hbox {\rm pole}} \ 
\end{align*}
Notice that this information, with multiplicities and the normalization of $1$ at $0$,
is enough to determine uniquely
$R_1\star R_2$.

\begin{proof}
It is just simple distributivity of the e\~ne product.
The formula for the degree follows from the formula
for the degrees of the e\~ne multiplication of polynomials.
\end{proof}


The previous observations do extend to meromorphic functions
on $\CC$ quotient of two entire functions, when $A\subset \CC$.

\begin{theorem} We have the same formula as before for the
e\~ne product. More precisely,
if $A\subset \CC$ and $f_1  , f_2 \in \cA$ are meromorphic
functions, quotient of entire functions of order $<1$ with
coefficients in $A$, then $f_1\star f_2$  is a meromorphic
function quotient of entire functions of order $<1$ given by the
above formula, and whose zeros are the products of zeros of $f_1$
and $f_2$, or the product of poles of $f_1$ and $f_2$, and whose
poles are the products of a zero of $f_1$ (resp. $f_2$) and a pole
of $f_2$ (resp. $f_1$). 
\end{theorem}

As we prove in section \ref{sec:enti} this result extends to arbitrary entire functions.

We assume for the rest of this section that $\QQ\subset A$.
The e\~ne product by $\exp (-X/(1-X))$ has an interesting property.

\begin{theorem}\textbf{(Convolution formula).}\label{thm:convolution}
For $f\in \cA$
$$
e^{-{X \over 1-X}} \star f=e^{X\cD (f)} =\cD_{exp} (f)\ .
$$
In particular, if $f$ is an entire function of order $<1$
with zeros $(\a_i)$ (in particular when $f$ is a polynomial),
then the e\~ne-multiplication by the function $\exp(-z /(1-z) )$
creates a function with essential isolated singularities at the $\a_i$'s :
$$
e^{-{z \over 1-z}} \star f=\exp \left ( \sum_i {z \over z-\a_i}
\right ) \ .
$$
\end{theorem}

\begin{proof}

This results from the main formula. We have
\begin{align*}
e^{X\cD (f)} &= e^{X\cD (f\star (1-X) )} \\
&= e^{X \cD (1-X)} \star f\\
&=e^{-{X \over 1-X}} \star f 
\end{align*}
\end{proof}


In the next theorem we have a list of computations of
various e\~ne products.

\begin{theorem}\textbf{ (Some computations).}
We have

\begin{itemize}
\item   For $a,b \in A$,
$$
{1\over 1-aX} \star {1\over 1-bX}=(1-aX)\star (1-bX) =1-abX  \ .
$$

\item   Let $P\in A[X]$ and $f\in \cA$ then
$$
e^{XP(X)}\star f=e^{XQ(X)} \ ,
$$
where $Q(X) \in A[X]$ is a polynomial with $\deg Q \leq \deg P$.

\item  For $N\geq 1$ positive integer,
let $E_N(X)$ denote the Weierstrass factor
\begin{align*}
 E_N(X) &=(1-X) \exp \left ( X+{X^2 \over 2}+\ldots +
{X^N \over N} \right )\\ 
&=\exp \left ( -{X^{N+1} \over N+1}
-{X^{N+2} \over N+2} -\ldots \right )
 \end{align*}
For $f\in \cA$  we have
$$
E_N\star f =f . T^e_N (1/f) \ ,
$$
where $T_N^e$ is the exponential N-truncation operator, i.e.
$$
T_N^e \left (\exp \left (\sum_{i=1}^{+\infty } F_i X^i \right )\right )=\exp \left (
T_N \left (\sum_{i=1}^{+\infty } F_i X^i \right )\right )=
\exp \left (\sum_{i=1}^{N } F_i X^i \right ) \ .
$$


\item  For $N,M \geq 1$ positive integers, we have
$$
E_N \star E_M =E_{\max (N,M)} \ .
$$


\item  For $N\geq 1$ we define
$$
I_N(X)=1-X^N \ .
$$
For $f(X) \in \cA$, $f(X)=\exp (\sum_{i=1}^{+\infty} F_i X^i)$, we
have
$$
I_N \star f(X)= \exp \left (\sum_{k=1}^{+\infty} F_{Nk} X^{Nk} \right ) \ .
$$


\item  For $N,M\geq 1$,
$$
I_N \star I_M =I_{{\rm l.c.m.} (N,M)} \ .
$$


\item  For $a\in A$ we define
$$
(1+X)^a=\sum_{n=0}^{+\infty} {a(a-1)\ldots (a-n+1) \over n!} X^n \ .
$$
For $f,g \in \cA$ and $a\in A$,
$$
f(X)^a\star g(X)=f(X)\star g(X)^a =(f(X)\star g(X))^a \ .
$$


\item  Action of the Artin-Hasse exponential. We recall
that
$$
\exp (X) =\prod_{n=1}^{+\infty} (1-X^n)^{\mu(n)/n}
$$
where $\mu$ is the M\"oebius function.
For a prime $p$, the Artin-Hasse exponential is
\begin{align*} 
\exp_p(X) &=\prod_{n=1 ; n \not=k p} ^{+\infty}
(1-X^n)^{\mu(n)/n} \\ 
&=\exp \left ( X+\frac{X^p}{p}+\frac{X^{p^2}}{p^2} +\ldots \right ) 
 \end{align*}

If $f\in\cA$, $f(X)=\exp \left ( \sum_{i=1}^{+\infty} F_i X^i \right )$,
then
$$
\exp_p(X) \star f(X)=\exp \left (  -\sum_{k=1}^{+\infty} F_{p^k}
X^{p^k} \right ) \ .
$$
\end{itemize}
 
\end{theorem}

\begin{proof}

We prove one of the formulas and left the others as exercices.
We have
\begin{align*}
E_N \star E_M &= \exp \left ( -{X^{N+1} \over N+1}
-{X^{N+2} \over N+2} -\ldots \right ) \star \exp \left ( -{X^{M+1} \over M+1}
-{X^{M+2} \over M+2} -\ldots \right )\\
&=\exp \left ( -{X^{\max(N, M)+1} \over \max(N, M)+1}
-{X^{\max(N, M)+2} \over \max(N, M)+2} -\ldots \right )\\
&=E_{\max (N,M)}  \ .
\end{align*}

\end{proof}

In view of the action of the action by e\~ne product of $I_n (X)=1-X^n$,
it is natural to define, in parallel with the theory of modular forms,
the following Hecke operators.

\begin{definition}\textbf{(Hecke operators).}
For $n\geq 1$ we define,
$$
T(n) : \cA \to \cA
$$
by
$$
T(n) (f) (X)=(I_n \star f) (X^{1/n}) \ ,
$$
that is, if $f(X)=\exp \left ( \sum_{i=1}^{+\infty} F_i X^i \right )$,
$$
T(n) (f) =\exp \left ( \sum_{k=1}^{+\infty} F_{nk} X^k \right ) \ .
$$
Note that $T(n)$ can be defined in the same way on $1+A[[X^{1/\lambda}]]$
for $\lambda \in \CC^*$.
\end{definition}

We can also define the "dilatation operators" by

\begin{definition}\textbf{(Dilatation operators).} For $\lambda \in \CC^*$
we define,
$$
R_\lambda : \cA \to 1+A[[X^{ 1/\lambda }]]
$$
by
$$
R_\lambda (f) (X)=f (X^{1/\lambda}) \ .
$$
Note that $R_\lambda$ is defined in the same way
on $1+A[[X^{1/\mu}]]$ for any
$\mu \in \CC^*$.
\end{definition}

We observe that $T(n)$ factors.

\begin{theorem}
We have
$$
T(n) (f)=R_{n} (I_n\star f) \ .
$$
\end{theorem}

Note that extending properly the e\~ne product to
$1+A[[X^{1/\lambda}]]$ we have commutation of $R_\lambda$ and
e\~ne multiplication by $I_n$, thus we can also write
$$
T(n) (f) =I_n \star (R_{1/n} (f)) \ .
$$
Then we have, similar (and simpler) formulas than in the theory
of modular forms (see \cite{S} p.159
for example),

\begin{theorem}
We have
\begin{itemize}
\item  For $\lambda , \mu \in \CC^*$,
$$
R_\lambda R_\mu =R_{\lambda \mu}
$$

\item  For $n\geq 1$ and $\lambda \in \CC^*$,
$$
R_{\lambda } T(n) =T(n) R_{\lambda} \ .
$$

\item  For $n,m \geq 1$, $n\wedge m=1$,
$$
T(n) \ T(m) =T(nm) \ .
$$
\end{itemize}

 \end{theorem}

\section{E\~ne ring structure for a field $A$.} \label{sec:ring}

We assume in this section that $\QQ\subset A$.
The e\~ne product of polynomials $P$ and $Q$ of respective
degrees $d_1$ and $d_2$ is  the polynomial $P\star Q$ 
of degree $d_1d_2$ (because the roots of $P \star Q$ counted
with multipicity are the products of a root of $P$ with a root of 
$Q$). Thus the e\~ne product does not respect the graduation by
degrees. But in exponential form it does. More precisely,
for $N\geq 1$, let $\cA_N \subset \cA$ be the subset of $\cA$
$$
\cA_N=\{ f \in \cA ; \exists P\in X A[X], \ \ \deg (P)\leq N, \ \ 
f=\exp (P) \} \ .
$$
Observe that the exponential truncation $T_N^e$ defines a surjective ring
homomorphism $T_N^e : \cA \to \cA_N$
$$
f \mapsto T^e_N (f)
$$
We denote $\cI_N$ its kernel, thus
$$
\cA_N \approx \cA/\cI_N \ .
$$
Obviously the inclusions $ \cA_N \hookrightarrow \cA_{N+1}$ are ring homomorphisms.

\begin{theorem}
The subset $\cA_N \subset \cA$ is a 
subring of the e\~ne ring $(\cA , . ,\star )$
with unit $(1-X)/E_N(X)$ and zero $E_N(X)$. Morevoer,
$\cA$ is the direct limit of the $\cA_N$'s
$$
\cA =\lim_{\longrightarrow } \cA_N \ .
$$
\end{theorem}

\begin{proof}

The e\~ne product preserves each $\cA_N$ as is immediate
from its simple exponential form. Also $(1-X)/E_N(X)$ is the unit
since 
$$
(1-X)/E_N(X)=\exp \left (-X-{X^2\over 2}-\ldots -{X^N \over N} \right ) \ .
$$
Also $E_N(X)$ is the zero since
$$
E_N(X)=\exp (\cO (X^{N+1})) \ .
$$
The direct limit is clear.
\end{proof}


When $A$ is a field the following theorem gives a description of 
the ideals of the e\~ne-ring.

\begin{theorem}
We assume that $A$ is a field. The maximal ideals of 
$\cA$ 
are 
$$
\cJ_n =\{ f\in \cA; f(X)=\exp (F_1 X+F_2 X^2 +\ldots ) \ \ 
{\text{such that  }} F_n=0 \} \ .
$$

In particular, $\cA_N$ is a quasi-local ring, i.e. it has a finite
number of maximal ideals.
 \end{theorem}

\begin{proof}

Given an ideal $\cJ \subset \cA$, we consider the set of 
integers
$$
\NN_{\cJ}=\{ n\geq 1 ; \exists f(X)=\exp \left (\sum_{i\geq 1} F_i X^i \right ) 
\in \cJ\ ,
\ \ {\hbox {\rm {with}}} \ \ F_n \not=0 \}
$$
If $\NN_{\cJ}=\emptyset $ then $\cJ =\{ 1 \}$. Otherwise we have

$$
\cJ\subset \bigcap_{n\notin \NN_{\cJ}} \cJ_n
$$
 \end{proof}

\section{E\~ne product and tensor product.} \label{sec:tens}

We assume in this section that $\QQ\subset A$.
We have the well known formal relation for $M\in M_n(A)$,
$$
\det (I-MX)=\exp \left ( -\sum_{k=1}^{+\infty}
{\hbox {\rm Tr}} (M^k) {X^k \over k} \right ) \ .
$$
We recall that given $M,N \in M_n(A)$ the tensor product
$M\otimes N \in M_{n^2}(A)$ is defined by
$$
(M\otimes N) (x\otimes y)=(Mx)\otimes (Ny)\ .
$$
In terms of the coefficients of the matrices
$$
(M\otimes N)_{(j,k) (i,l)} =M_{ij} N_{kl} \ .
$$
Thus, in particular, we have
$$
{\hbox {\rm Tr}} (M\otimes N)={\hbox {\rm Tr}} (M) . {\hbox {\rm Tr}} (N) \ .
$$
From these observations and the exponential form of the e\~ne product
we get the following Theorem:

\begin{theorem}\label{thm:6.1}
For $M,N \in M_n(A)$ we have
$$
\det (I-M X) \star \det (I-N X)=\det (I-(M\otimes N) X) \ .
$$
\end{theorem}
%

This last result provides a linear algebra procedure to compute
the e\~ne product. Notice that if $P(X)=1+a_1 X+a_2X^2+\ldots +a_dX^d$
we have
$$
P(X)=\det \left [
\begin{array}{cccccc}
1 & 0 & 0 & \cdots & 0 & (-1)^{d-1} a_d X \\
        -X & 1 & 0 & \cdots & 0 & (-1)^{d-2} a_{d-1} X \\
         0 & -X & 1& \cdots & 0 & (-1)^{d-3} a_{d-2} X \\
    \vdots &    & \ddots &\ddots &\vdots &\vdots \\
    \vdots &    &        & \ddots & 1 &-a_2 X \\
         0 & \cdots & \cdots &  0   & -X & 1+a_1X 
\end{array}
\right ] =\det (I-M_P X)
$$

where
$$
M_P=\left [
\begin{array}{cccccc}
0 & 0 & 0 & \cdots & 0 & (-1)^{d} a_d  \\
        1 & 0 & 0 & \cdots & 0 & (-1)^{d-1} a_{d-1}  \\
         0 & 1 & 0& \cdots & 0 & (-1)^{d-2} a_{d-2}  \\
    \vdots &    & \ddots &\ddots &\vdots &\vdots \\
    \vdots &    &        & \ddots & 0 &a_2  \\
         0 & \cdots & \cdots &  0   & 1 & -a_1 
\end{array}
\right ]
$$

Thus we get:

\begin{theorem}\label{thm:6.2}
We have
$$
P (X)\star Q (X) =\det (I-(M_P \otimes M_Q)X) \ .
$$
\end{theorem}

Notice that the extension of the e\~ne product to formal power series
indicates that theorem \ref{thm:6.2} remains valid for infinite matrices.
Also theorem \ref{thm:6.1} makes sense for infinite matrices once the tensor
product and the determinant are properly defined (one can also
define the infinite determinant the other way around).

\section{Analytic properties of the e\~ne product.} \label{sec:anal}

The e\~ne product satisfies remarkable analytic properties.
We assume in this section that $A\subset \CC$ and we study
the convergence properties of series.
Recall Hadamard formula for the radius of convergence of
$f\in \cA$, $f(z)=1+\sum_{i=1}^{+\infty} f_i z^i $,
$$
{1\over R(f)} = \limsup\limits_{i\to +\infty} |f_i|^{1/i} \ .
$$
It is convenient to introduce the e\~ne radius of convergence
of $f$ as
$$
\tilde R(f)=\min_i ( |\a_i| , R(f) )
$$
where $(\a_i)$ are the zeros of $f$. Since $f(0)=1$ we have
$$
R(f) \geq \tilde R(f) >0 \ .
$$
We observe that if we write $f=\exp (F)$ with $F=\log f$, then
$$
R(F)=\tilde R(f) \ .
$$
The first basic result is that the
e\~ne product of series with positive radius of convergence
is a series with positive radius of convergence :

\begin{theorem}\label{thm:anal1}
Let $f, g \in \cA$. We have
$$
R(f\star g) \geq \tilde R(f) . \tilde R(g) \ .
$$
In particular, the e\~ne product of two series with positive
radius of convergence has positive radius of convergence.
 \end{theorem}

\begin{remark}

We will improve this result and show that indeed
$$
\tilde R(f\star g) \geq \tilde R(f) . \tilde R(g) \ .
$$
\end{remark}

\bigskip

\begin{proof}

We write $f$ and $g$ in exponential form
\begin{align*}
f(z) &= \exp (F(z)) =\exp \left (\sum_{i=1}^{+\infty } F_i z^i\right ) \\
g(z) &= \exp (G(z)) =\exp \left (\sum_{i=1}^{+\infty } G_i z^i \right )
\end{align*}
Using Hadamard formula we get
\begin{align*}
{1\over R(F\star_e G)} &=\limsup\limits_{i\to +\infty} (i|F_i| |G_i|)^{1/i} \\
&=\limsup\limits_{i\to +\infty} (|F_i| |G_i|)^{1/i} \\
&\leq \left ( \limsup\limits_{i\to +\infty} |F_i|^{1/i} \right ).
\left ( \limsup_{i\to +\infty} |G_i|^{1/i} \right ) \\
&={1\over R(F)}.{1\over R(G)} 
\end{align*}
Therefore, we get
$$
R(f\star g)\geq R(F\star_e G) \geq R(F).R(G) =\tilde R(f) . \tilde R(g) \ .
$$
 
\end{proof}


We have the following continuity property:

\begin{theorem} \label{thm:anal2}
We consider the space
$\tilde \cA_{R_0}$ of power series
$f\in \cA$ with $\tilde R(f)\geq R_0$, i.e convergent and
with no zeros on the disk $\DD_{R_0}$
of center $0$ and radius $R_0 >0$. We consider also the space
$\cA_{R_0}$ of power series
$f\in \cA$ with $R(f)\geq R_0$. We
endow this spaces with the topology of
uniform convergence on compact subsets of $\DD_R$.
The e\~ne product $ \star :\tilde \cA_{R_1} \times  \tilde \cA_{R_2} \to \cA_{R_1 R_2}$
$$
(f,g)\mapsto f\star g
$$
is continuous.
 
\end{theorem}

\begin{proof}

Any function $f \in \tilde \cA_{R_0}$ can be written
$f=e^F$ with $F=\log f$ having radius
of convergence at least $R_0>0$. The linear expression of the e\~ne product
on the coefficients of $F$ shows the continuity.
\end{proof}


We can now improve Theorem \ref{thm:anal1}.

\begin{corollary}
Let $f, g \in \cA$. We have
$$
\tilde R(f\star g) \geq \tilde R(f) . \tilde R(g) \ .
$$
 \end{corollary}

\begin{proof}
We only need to show that any zero $\xi$ of $f\star g$ with
$|\xi | < R(f\star g)$ satisfies $|\xi | \geq \tilde R(f) \tilde R(g)$.
If $|\xi | \geq R(f) R(g) \geq \tilde R(f) \tilde R(g)$ we are done.
We assume then $|\xi | < R(f) R(g)$. By the previous theorem
we have that
$$
T_N(f) \star T_N(g) \to f\star g
$$
when $N\to +\infty$ uniformly on compact sets in
$\DD_{R(f)R(g)}$. Thus, in particular,
we have uniform convergence in a compact neighborhood of $\xi$.
Then by Hurwitz theorem $\xi$ must be the limit of zeros of
$T_N(f)\star T_N(g)$. Any such zero is of the form $\a_N \b_N$ where
$\a_N$ (resp. $\b_N$) is a zero of the polynomial $T_N(f)$ (resp.
$T_N(g)$). Since $\a_N \b_N \to \xi$ we must have that the sequences
$(\a_N)$ and $(\b_N)$ are bounded (if $\a_{N_k}\to \infty$ then
$\b_{N_k} \to 0$ which is impossible). We can extract converging
subsequences $\a_{N_k} \to \a$ and $\b_{N_k} \to \b$. Finally we
observe that $|\a|\geq \tilde R(f)$ and $|\b|\geq \tilde R(g)$.
Because if $|\a| < R(f)$ then since $T_N(f)\to f$ in   $\DD_{R(f)}$
then $\a$ would be a zero of $f$ thus $|\a | \geq \tilde R(f)$.
The same argument applies to $\b$. We conclude
$$
|\xi|=|\a| |\b| \geq  \tilde R(f)\tilde R(g) \ ,
$$
as we wanted to show. 
\end{proof}

\begin{remark}
 We show in section \ref{sec:hada} as an application of Hadamard multiplication theorem 
that we do have the equality
$$
\tilde R(f\star g) = \tilde R(f) . \tilde R(g) \ .
$$

\end{remark}

\section{E\~ne product and entire functions.} \label{sec:enti}

We assume in this section that $A\subset \CC$. It is not difficult to see,
from the interpretation involving the zeros, 
that the e\~ne product
does extend from polynomials  to entire functions of order $<1$ and 
leaves this space invariant. The next result shows that we have 
better, Weierstrass factors cause no trouble and the e\~ne product
extends to functions of finite order.

\begin{theorem}
We consider
$0\leq \lambda <+\infty$.
We define the space $\cE_{\lambda} \subset \cA$ of entire 
functions of order $<\lambda$ with constant coefficient $1$.

The e\~ne product is an internal operation in $\cE_{\lambda}$
and $(\cE_{\lambda}, .,\star )$ is a subring of $\cA$.

Moreover, the e\~ne product
$\star : \cE_{\lambda} \times \cE_{\lambda} \to \cE_{\lambda}$, 
$(f,g)\to f\star g$
is continuous for the topology of uniform convergence on compact 
subsets.
\end{theorem}
This theorem results from the next result that is 
more general and 
that shows that the e\~ne product
respects Hadamard-Weierstrass factorization
of entire functions of finite genus. 
We recall that the genus $\rho$ 
of en entire function $f$ is the minimal integer so that 
$f$ can be written
$$
f(z)=e^{F(z)} \prod_i E_{\rho}\left ({z\over \a_i} \right ) 
$$
where $F\in \CC [z]$ is a polynomial of degree $\leq \rho$, 
and the infinite product is uniformly convergent on compact subsets 
of $\CC$. The factorization for general entire functions is
due to Weierstrass. The above factorization for functions of 
finite genus is due to Hadamard.
If no such $\rho$ exists then the genus is infinite.
We have $\rho\leq \lambda\leq \rho+1$ (see for example \cite{A} p. 209).

\begin{theorem}
Let $f$ and $g$ be entire functions of 
finite genus $0\leq \rho <+\infty$ with respective
sets of zeros $(\a_i)$ and $(\b_j)$. We assume that $f(0)=g(0)=1$.
We consider the Hadamard-Weierstrass
factorizations 
\begin{align*}
f(z) &=e^{F(z)} \prod_i E_{\rho }\left ({z\over \a_i} \right )  \\
g(z) &=e^{G(z)} \prod_j E_{\rho }\left ({z\over \b_j} \right )  
\end{align*}
where, by definition of the genus, 
$F$ and $G$ are polynomials vanishing at $0$ of degree $ \leq \rho$.
Then $F \star_e G $ is a polynomial of degree $\leq \rho$
and we have
$$
f\star g (z)=e^{F \star_e G (z)} \prod_{i,j} 
E_{\rho }\left (z\over \a_i \b_j \right ) \ .
$$
\end{theorem}

\begin{proof}
The exponential form of the e\~ne product shows
that $F \star_e G $ is 
a polynomial of degree $\leq \rho$. 

Now, working on the ring $\cA$ (thus we do not need to pay 
attention to questions of convergence for the moment 
and the computations are done 
at the formal level), we have, using distributivity,
\begin{align*}
f\star g (z) &=(e^F \star e^G). \left (e^F \star
\prod_j E_{\rho }\left ({z\over \b_j} \right ) \right ) . 
\left ( \prod_i E_{\rho }\left ({z\over \a_i} \right )
\star  e^G \right ). \\
&\ \ \ \ \ \ \left ( \prod_i E_{\rho }\left 
({z\over \a_i} \right ) \star \prod_j E_{\rho }\left ({z\over \b_j} \right ) 
\right )
\end{align*}
Now, we have
$$
e^F \star e^G =e^{F\star_e G} \ .
$$
Also  we have
\begin{align*}
e^F \star
\prod_j E_{\rho }\left ({z\over \b_j} \right )&=\prod_j e^F \star 
E_{\rho }\left ({z\over \b_j} \right ) \\ 
&=\prod_j e^{F(z/\b_j)} \star 
E_{\rho }(z ) \\
&=\prod_j e^{F(z/\b_j)}. T_\rho^e \left (e^{-F} \right )(z/\b_j)  \\
&=\prod_j e^{ \left ( F-T_\rho (F) \right ) (z/\b_j)} \\
&=1 
\end{align*}
because $\deg F \leq \rho$ thus $F-T_\rho (F) =0$.
By the same reasons
$$
\prod_i E_{\rho }\left ({z\over \a_i} \right )
\star  e^G =1 \ .
$$
And finally, 
\begin{align*}
& \prod_i E_{\rho }\left 
({z\over \a_i} \right ) \star \prod_j E_{\rho }\left ({z\over \b_j}\right ) \\
&=\prod_{i,j} E_{\rho }\left 
({z\over \a_i} \right ) \star E_{\rho }\left ({z\over \b_j} \right ) \\
&=\prod_{i,j} E_{\rho }(z) 
\star E_{\rho }\left ({z\over \a_i \b_j} \right ) \\
&=\prod_{i,j} E_{\rho} \left ({z\over \a_i \b_j} \right ) . 
T_N^e \left ( 1/E_{\rho} \left ({z\over \a_i \b_j} \right ) \right ) \\
&= \prod_{i,j} E_{\rho} \left ({z\over \a_i \b_j} \right ) 
\end{align*}
where the last equality is obtained observing that 
$$
T_N^e \left ( 1/E_{\rho} \left ({z\over \a_i \b_j} \right ) \right ) =1 \ .
$$
Thus we have established the formal Weierstrass factorization
for $f\star g$. We only need to check that the product 
of Weierstrass factors is uniformly convergent on compact subsets 
of $\CC$. This follows from the continuity of the e\~ne product
for the topology of uniform convergence on compact sets
on a domain where the functions have no zeros. Given a compact set 
in the plane, we consider a ball $\DD_R$ of center $0$ and radius $R>0$ 
large enough to contain the compact set.
Consider only those Weierstrass factors having zeros out 
of this ball, we observe that their product converges uniformly 
as well as they e\~ne product by theorem \ref{thm:anal2}. The remaining
Weierstrass factors are finite.
\end{proof}

\begin{remark}
This previous result has a generalization to arbitrary 
entire functions of infinite genus. We must then choose 
the orders in the Weierstrass factors large enough (depending
on $f$ and $g$) in order not to introduce other terms in 
the exponential besides $F\star_e G$.
 \end{remark}

\section{E\~ne-product and Hadamard multiplication.} \label{sec:hada}

We consider in this section an arbitrary ring $A$ unless otherwise stated.
We recall the definition of Hadamard multiplication (see \cite{Ha}).

\begin{definition}\textbf{(Hadamard multiplication)}
The Hadamard multiplication of $f(X), g(X) \in 
A[[X]]$,
\begin{align*}
f(X) &=\sum_{n=0}^{+\infty } f_n X^n \\
g(X) &=\sum_{n=0}^{+\infty } g_n X^n
 \end{align*}
is
$$
f\odot g (X)=\sum_{n=0}^{+\infty} f_n g_n X^n \ .
$$
\end{definition}

Note that the Hadamard multiplication is an internal operation in
$A[X]$, $XA[[X]]$, $\cA$ and $A[[X]]$. The neutral element in $A[[X]]$ for the Hadamard multiplication is
$$
{1 \over 1-X} =1+\sum_{n=1}^{+\infty} X^n \ .
$$

More precisely, we have

\begin{theorem}
The sum and the Hadamard multiplication are internal 
operations in $A[X]$, $\cA$ and $A[[X]]$, and
$(A[[X]], + , \odot )$ is a unitary commutative ring.
\end{theorem}

The Hadamard multiplication has similar properties than the exponential e\~ne product $\star_e$.
For example, we have

\begin{theorem}
If $P(X)\in A[X]$ and $f(X)\in A[[X]]$ then
$P\odot f \in A[X]$, and 
$$
\deg (P) =\deg (P \odot f)\ .
$$
\end{theorem}

The relation to the exponential e\~ne product is clear
from the definition.

\begin{theorem}
We have for $F(X), G(X) \in A[[X]]$,
$$
F\star_e G=- K_0\odot F \odot G 
$$
where 
$$
K_0 (X) =X +2 X^2+3X^3+\ldots ={X \over (1-X)^2} 
$$
is the Koebe function.
\end{theorem}
The Koebe function plays a central role in Univalent
Function Theory, being extremal for many problems. This result
means that the exponential e\~ne product structure is the
Hadamard ring structure twisted by $-K_0$.
Note that when $\QQ\subset A$, the inverse of $-K_0$ for
the Hadamard multiplication in $XA[[X]]$  is
$$
-X-{1\over 2} X^2-{1\over 3} X^3-\ldots =\log (1-X) \ ,
$$
i.e. it is also the unit for the exponential e\~ne product.
Directly from the definition we get:

\begin{theorem}
Let $F,G \in A[[X]]$. We have
$$
D(F\star_e G)=-D(F)\odot D(G) \ .
$$
More precisely,
$-D: (XA[[X]], +, \star_e) \to (A[[X]], +, \odot)$, $F\mapsto -D(F)$
is a ring homomorphism between the exponential e\~ne ring
structure and the Hadamard ring structure. It is an isomorphism when $\QQ\subset A$.
\end{theorem}

As corollary  we get the direct relation to 
the e\~ne product.

\begin{theorem}
Let $f,g \in \cA$. We have
$$
\cD (f \star g )=-\cD (f) \odot \cD (g) \ .
$$
\end{theorem}

\begin{proof}
Write 
\begin{align*}
f&=\exp (F) \\
g &=\exp (G) 
\end{align*}
and observe that 
$$
\cD (f \star g ) =D(F\star_e G) =-D(F) \odot D(G) =-\cD (f) \odot 
\cD (g) \ .
$$
\end{proof}

\section{Extension of the e\~ne product and inversion.} \label{sec:exte}

We assume that $A$ is a field.
We extend the definition of e\~ne product to the full ring of
non zero polynomials $A[X]$ by using the interpretation 
with roots. 

\begin{definition}
For $P(X), Q(X) \in A[X]$, $P$ and $Q$ non zero, 
with 
\begin{align*}
P(X)&=a_0 X^n P_0(X) \\
Q(X)&=b_0 X^m Q_0(X) 
\end{align*}
where $a_0, b_0 \in A-\{ 0\}$, and $P_0(X), Q_0(X) \in 1+X \ A[X]$ we define
$$
(P\star Q)(X)=X^{n \deg (Q_0) +m \deg (P_0) +nm} (P_0 \star Q_0)(X) \ .
$$
\end{definition}

For simplicity we assume that $A$ is a field.
Denote by $\PP A[X]$ the projective space of non-zero polynomials
(two polynomials differing by a non-zero multiplicative constant are
equivalent).

\begin{proposition}
We have that $ (\PP A[X], ., \star)$ is
a commutative ring.
 \end{proposition}

We can extend the e\~ne product to the projective space of rational functions $\PP A(X)$.

\begin{definition}
We extend the ene product to non-zero rational
functions quotients of elements in $\PP A(X)$. If we have
\begin{align*}
R_1(X) &= {P_1(X) \over Q_1(X) } \\
R_2(X) &= {P_2(X) \over Q_2(X) } 
\end{align*}
with $P_1, P_2, Q_1, Q_2 \in \PP A[X]$ then we define
$$
(R_1 \star R_2)(X)={(P_1 \star P_2).(Q_1 \star Q_2) \over (P_1\star Q_2).
(Q_1 \star P_2)} \ .
$$
\end{definition}

\begin{proposition}
We have that $ (\PP A(X), ., \star)$ is
a commutative ring.
 \end{proposition}
Next we prove  a main property of this extension of the e\~ne
product. It shows that the points at $0$ and $\infty$ play a symmetric  role.

\begin{theorem}
The e\~ne product is invariant by inversion.
More precisely, let $P(X), Q(X) \in \PP A[X]$, then
$$
P({1/ X})\star Q({1/ X}) =(P\star Q) ({1/ X}) \ .
$$
\end{theorem}

\begin{proof}
Write 
\begin{align*}
P({1/X})&=\frac{\hat P (X)}{ X^n} \ ,\\
Q({1/ X})&=\frac{\hat Q (X)}{X^m} \ ,
\end{align*}
where $\hat P(X), \hat Q(X) \in \cA$. In the case $A= \CC$, we
observe that if $(\a_i)$ are the zeros of $P$ then the
zeros of $\hat P$ are $(\a_i^{-1} )$. From this observation
and universality of the polynomial formulas it follows that
$$
(P\star Q)({1/X})=\frac{(\hat P \star \hat Q )(X)}{X^{nm}} \ .
$$

We have now
\begin{align*}
P({1/X})\star Q({1/X}) &= \frac{\hat P (X)}{X^n} \star
\frac{\hat Q(X)}{X^m} \\
&=\frac{(\hat P \star \hat Q).(X^n\star X^m)}{(\hat P \star X^m).
(\hat Q \star X^n)} \\
&=\frac {\hat P\star \hat Q}{X^{nm}} \\
&= (P\star Q) ({1/X})
\end{align*}

\end{proof}

We have an easy application (compare with the classical proofs):

\begin{proposition}
The set of non-zero algebraic numbers in $\CC$, resp. algebraic integers in $\CC$,  is a multiplicative group.
More precisely, if $\alpha$ and $\beta$ are algebraic numbers, resp. algebraic integers, then $\alpha \beta$ is an algebraic number, resp. algebraic integer.
\end{proposition}

\begin{proof}
If $\alpha =0$ or $\beta=0$ the result is clear. Otherwise, choose $P, Q \in \QQ[X]$ such that $P(0)=1$ and $Q(0)=1$ such that
$P(\alpha)=Q(\beta)=0$. Then we have that $P\star Q \in \QQ[X]$ and $P\star Q(\alpha \beta)=0$ which proves that $\alpha$ and $\beta$ are algebraic numbers.

If $\alpha$ and $\beta$ are non-zero algebraic integers,
then choose monic polynomials
$P\in \ZZ[X]$ of degree $n\geq 1$, and $Q\in \ZZ[X]$ of degree $m\geq 1$ such
that $P(\alpha)=0$ and $Q(\beta)=0$. Consider the polynomials
$\hat P(X)= X^nP(1/X)$ of degree $n$ and
$\hat Q(X) =X^m Q(1/X)$ of degree $m$. Then $\hat P(1/\alpha)=0$, $\hat Q(1/\beta)=0$ and
$\hat P(X), \hat Q(X) \in 1+X\ZZ[X]$. Therefore $\hat P\star \hat Q \in 1+X\ZZ[X]$ and $(\hat P \star \hat Q) (1/(\alpha \beta))=0$ and the polynomial
$X^{n+m}(\hat P \star \hat Q) (1/X)$ is monic, has integer coefficients
and annihilates $\alpha.\beta$, so $\alpha.\beta$ is an algebraic integer.

\end{proof}

\bigskip

\textbf{Acknowledgements.} The author is grateful to  Daniel Barsky, Hendrik Lenstra, and Jean-Pierre Ramis for corrections in the first version of this article.

\end{document}